\documentclass[oneside]{amsart}

\usepackage[T1]{fontenc}

\usepackage[latin1]{inputenc}

\usepackage{amsmath,amsthm,amsfonts,amssymb,latexsym, mathrsfs}

\usepackage{tabularx}
\usepackage{graphicx, MnSymbol} %%For loading graphic files
\usepackage{subfig} %%Subfigures inside a figure
\usepackage{tikz}
\usetikzlibrary{matrix, positioning}

\makeatletter

\theoremstyle{plain}
\makeatletter

\usepackage{amssymb}

\usepackage{amsmath}

\def \<{\langle}
\def \>{\rangle}

\usepackage{amsfonts}

\usepackage{amscd}

\usepackage{latexsym}

\usepackage{epsfig}

\usepackage{Mycommands}

\begin{document}

\title{Isometry group of Borel randomizations}

 \date{\today}

\author{Alexander Berenstein}
\address{Universidad de los Andes,
Cra 1 No 18A-10, Bogot\'{a}, Colombia}
 \urladdr{www.matematicas.uniandes.edu.co/\textasciitilde aberenst}

\author{Rafael Zamora}
\address{Escuela de Matem\'atica, sede Rodrigo Facio, Universidad de Costa Rica, Montes de Oca, San Jose, Costa Rica}
\email{rafael.zamora\_c@ucr.ac.cr}

\keywords{Topological groups, isometry groups, Borel randomizations, Rohklin property, topometric spaces, topometric groups.}
\subjclass[2010]{Primary: 03E15, 54H11 Secondary: 22F50, 22A05, }
\thanks{We would like to thank Julien Melleray for suggesting the topological approach to the problem presented in section \ref{Den2}. We would also like to thank John K. Truss for the information he provided to us which allowed us to extend our result to $\mbox{Aut}(Q,<)$. The second author would like to thank the department of mathematics at Universidad de los Andes for its hospitality during his stay, which resulted in part of this work.}

\begin{abstract}
We study global dynamical properties of the isometry group of the Borel randomization of a separable 
complete structure. In particular, we show that if properties such as the Rohklin property, topometric generics, extreme amenability hold for the isometry group of the structure, they also hold in the isometry group of the randomization.
\end{abstract}

\pagestyle{plain}

\maketitle

\section{Introduction}\label{Intro}
This paper deals with structural properties of isometry groups of randomizations of metric structures (see \cite{BK}), in particular the existence of generic elements. 

The setting is the following. Given $G$ a Polish group, we say that an element $g\in G$ is generic if its orbit under conjugation $\{g^h:h\in G\}$ is comeager and we say that $G$ has \emph{generics} if it has a generic element. We say that an $n$-tuple $(g_1,\dots,g_n)$ is generic if its orbit under the action of $G$ by pointwise conjugation $\{(g_1^h,\dots,g_n^h) : h\in G\}$ is comeager. Finally we say that $G$ has ample generics if for every $n\geq 1$, $G$ has generic $n$-tuples.

Ample generics where introduced by Hodges, Hodkinson, Lascar and Shelah in \cite{HHLS} and some of its consequences were explored by Kechris and Rosendal in \cite{KR}. Among other properties, they showed that if a Polish group has ample generics, then it has the automatic continuity property, namely, any homomorphism from $G$ to a Polish group $H$ is always continuous.

The examples studied in \cite{KR} are subgroups of $S_\infty$ that arise as automorphism groups of Fra\"iss\'e limits whose partial automorphisms have nice amalgamation properties. They include  the permutation group $S_\infty$ of $\mathbb{N}$, the automorphism group of the random graph, the automorphism group of a countably dimensional vector space over a finite field and the automorphism group of the rational Urysohn space. A group can have generics and fail to have ample generics (for example $Aut(Q,<)$). All these examples are totally disconnected.

The work of Ka\"ichouh and Le Ma\^itre \cite{KL} shows how to build connected examples. Let $L^0([0,1],G)$ denote the space of measurable functions from $[0,1]$ to $G$, which is a group with pointwise multiplication and Polish with the topology of convergence in measure. They prove that if $G$ has generics (respectively ample generics), then $L^0([0,1],G)$ has generics (respectively ample generics).

On the other hand, several Polish groups that arise as the group of isometries of metric structures
do not have generics or ample generics but weaker properties called \emph{metric generics} and \emph{metric ample generics} that were introduced and studied in \cite{BBM}. This is the case for the group of isometries $Aut([0,1])$ of the measure algebra associated standard Lebesgue space, the group of isometries of the Urysohn space and the group of isometries of a separable Hilbert space \cite{BBM}. 
The main idea behind this approach is to endow a group of isometries $G$ of a metric space (or even a metric structure) $(M,d)$ with two topologies, the one of pointwise convergence (which is Polish) and the one of uniform convergence (which is finer than the previous one and in general is not Polish). The interaction of the Polish topology and the uniform convergence topology gives a notion of relative continuity for group homomorphisms that replaces the usual notion of automatic continuity (see \cite{BBM}). In this paper we apply the ideas of \cite{BBM}) to the isometry group of a randomization of a countable first order structure $M$.
Ibarluc\'ia \cite{Ib} characterized this group in terms of the group of isometries of $M$. He showed that it can be writen as $\tilde G=L^0([0,1],G) \rtimes Aut([0,1])$ where $G=Isom(M)$ and for 
$\alpha \in L^0([0,1],M)$ and $(f,T)\in L^0([0,1],G) \rtimes Aut([0,1])$, the action is given by $((f,T)(\alpha))(\omega)=f(\omega)\alpha(T^{-1}(\omega))$.

The group $\tilde G$, being a group of isometries, can be endowed with the topology of pointwise convergence and the topology of uniform convergence. Ibarluc\'ia \cite{Ib} proved that the pointwise convergence topology is the product topology of $L^0([0,1],G)$ and $Aut([0,1])$. We show a similar result 
and prove that the uniform convergence topology is the product of an induced uniform topology in $L^0([0,1],G)$ and the uniform topology in $Aut([0,1])$. With these tools we show that if $G$ has metric generics
(respectively metric ample generics), then $\tilde G$ has metric generics
(respectively metric ample generics). We also prove that if $G$ is extremely amenable, then so is $\tilde G$.

As a corollary of our work we can also show that if $H$ and $G$ are isometry groups and both have metric ample generics,
then when we equip $H \rtimes G$ with the Polish product topology and the with the uniform convergence 
topology coming from the product of the corresponding uniform topologies, then $H \rtimes G$ also has metric ample 
generics. Similarly we prove that if the Rohklin property holds for $H$ and $G$ then it also holds for $H \rtimes G$.

We should point out that many model theoretic properties transfer from $T=Th(M)$ to the theory of its randomization $T^R$ such as $\omega$-categority, $\omega$-stability and NIP \cite{BK,BY}. The work of 
Ibarluc\'ia \cite{Ib} also shows that some properties of $G$ are reflected in $\tilde G$ such as 
being reflexively representable. This, together with the results of  Ka\"ichouh and Le Ma\^itre, was a strong indication that we also could expect some nice behavior at the level of generics in the automorphism group.

This paper is organized as follows. In section \ref{KlM3} we follow the work of \cite{KL} and show that if
$G$ is a Polish group and has metric generics (respectively metric ample generics), so does $L^0([0,1],G)$.
In section \ref{uct} we characterize the uniform convergence topology in $\tilde G$ when we see it as the isometry group of a randomization. In section \ref{Den} we approach the problem of existence of dense orbits in $\tilde S_\infty$. We handle the problem from an algebraic perspective and we extend the result to a robust class of Polish groups. In section \ref{Den2}, we follow a more topological approach and show that dense orbits are transferred from $G$ to $\tilde G$ as well as topometric ample generics. In section \ref{EAm} we show that extreme amenability transfers from $G$ to $\tilde G$.

Finally, we should point out that there are other approaches to study automatic continuity in isometry groups such as those in the work of Kwiatkowska, Malicki and Sabok \cite{KM,Ma,Sa}. The notes by  Ka\"ichouh \cite{Ka} are a good reference for the subject.

\section{Transferring topometric generics to $L^0([0,1],G)$}\label{KlM3}

In this section we generalize the ideas of Ka\"ichouh-Le Maitre and show how the map that sends $G$ to $L^0([0,1],G)$ not only preserves generics and ample generics but also metric generics and ample metric generics. We start by recalling some definitions from \cite{BBM}.

\begin{defi}
We say that a triple $(X,\tau,d_u)$ is a \textbf{topometric space} if $X$ is a set, $\tau$ is a topology on $X$ and $d_u$ is a distance function such that:
\begin{enumerate} 
\item The topology induced by $d_u$ refines $\tau$.
\item The metric $d_u$ is lower semi-continuous with respect to $\tau$, i.e., the set $\{(x,y)\in X^2: d_u(x,y)\leq r\}$ is $\tau$-closed for all $r\geq 0$.
\end{enumerate}
We say that a triple $(G,\tau,d_u)$ is a \textbf{topometric group} if $(G,\tau)$ is a topological group,
$(G,\tau,d_u)$ is a topometric space and $d_u$ is bi-invariant under group multiplication.

In the setting above, whenever $(X,\tau)$ is a Polish space, we say that $(X,\tau,d_u)$ is a \textbf{Polish topometric space} and if $(G,\tau)$ is a Polish group, we say that $(G,\tau,d_u)$ is a \textbf{Polish topometric group}.
\end{defi}

\begin{defi}
Let $(X,\tau,d_u)$ be a Polish topometric space and let $d$ be a complete metric inducing the topology $\tau$ such that there is a constant $C$ with $d(x,y)\leq C d_u(x,y)$. Then we say that $d_u$ is a distance that \textbf{$C$-uniformly refines} the metric space $(X,d)$.
\end{defi}

Note that this definition is equivalent to the identity function $i:(X,d_u)\to (X,d)$ being Lipschitz. Abusing notation, when we have a fixed metric $d$ for $\tau$ we will write $(X,d,d_u)$ instead of $(X,\tau,d_u)$.
Also note that if $(X,d,d_u)$ is a topometric space, so is $(X,d',d_u')$, where
$d'=\min\{d,1\}$ and $d_u'=\min\{d_u,1\}$ and the topologies do not change.  Thus we may always assume that both metrics are bounded.

\begin{exam}
Let $(X,\tau)$ be Polish, let $d_X$ be a complete metric on $X$ that induces the topology. Then 
$(X,\tau,d_X)$ is a topometric Polish space.
\end{exam}

\begin{exam}\label{du}
Let $(X,\tau)$ be Polish, let $d$ be a complete metric on $X$ that induces the topology and assume that 
$diam_{d}(X)\leq 1$. Let $\{b_m\}_m$ be a dense subset of $X$. Let $G=Isom(X)$ be the group of isometries of $X$ and for $g,h\in G$ let
$d_p(f,h)=\sum_m \frac{1}{2^{m+1}} d(f(b_m),h(b_m))$. Then $d_p$ induces the pointwise convergence topology on $G$ which is Polish.
Let $d_u(f,h)=\sup_{x\in X} d(f(x),h(x))$ which is a metric for the uniform convergence topology. Then $d_u$ $1$-uniformly refines $d_p$,  and $(G,d_p,d_u)$ is a topometric space. Note that $d_u$ is bi-invariant under multiplication by $G$, $d_u$ is lower semicontinuous with respect to $d_p$  and so $(G,d_p,d_u)$ is a Polish topometric group.
\end{exam}

%Let $(X,\tau_X)$ be Polish, let $d_X$ be a compatible complete metric on $X$ defining the topology
%and assume that $Diam(X)=1$. Let $G=Isom(X)$, which is Polish with the pointwise convergence topology which we %will denote by $\tau$. 

\begin{lem}
Assume that $(X,\tau)$ is a Polish space. Let $f,g\in L^0([0,1],X)$, and $d:X\times X\to \mathbb{R}$ a measurable function. Then the function sending $\omega$ to $d(f(\omega),h(\omega))$ is measurable. Morever, if $d$ is bounded, it is integrable in $[0,1]$. 
\end{lem}

\begin{proof}
Since $d$ is measurable, for any $r>0$, $F_r=\{(x,y)\in X\times X: d(x,y)<r\}$ is measurable.
But since $f,h$ are measurable, $\{\omega\in [0,1]:d(f(\omega),h(\omega))\leq r\}=\{\omega\in [0,1]:(f(\omega),h(\omega))\in F_r\}=(f,h)^{-1}(F_r)$ is measurable and thus $d(f(\omega),h(\omega))$ is a measurable function on $\omega$. Since $d$ is bounded and measurable, then it is integrable and the integral is finite (see \cite{WZ}). 
\end{proof}

Note that lower semicontinuous metrics are measurable. 

Throughout this paper, we will work out with several induced distances. We will now  standarize the notation.

\begin{nota}
\begin{enumerate}
	\item Let $(X,d)$ be a metric space. If $G=\mbox{Isom}(X,d)$ is the group of isometries, we will usually denote by $\tau$ the pointwise-convergence topology on $G$. However, if there are several topologies that need to be considered, we will write $\tau_p$ for the topology of pointwise convergence.

 \item We will denote by $d_u$ the uniform convergence distance on $\mbox{Isom}(X,d)$ as introduced in Example \ref{du}.
 \item Let $d$ be a bounded metric on $X$ and for $f,h\in L^0([0,1],X)$ let $\hat{d}(f,h)=\int_0^1 d(f(\omega),h(\omega)) d\omega$. Note that by the argument above $d(f(\cdot),g(\cdot))$ is integrable in $[0,1]$, thus $\hat d$ is well defined.
%Let $\delta:X\times X \to \mathbb{R}$ be a bounded measurable function and for $f,h\in L^0([0,1],X)$ let $\hat \delta(f,h)=\int_0^1 \delta(f(\omega),h(\omega)) d\omega$. Note that by the argument above $\delta$ is integrable in $[0,1]$, thus $\hat \delta$ is well defined.
\item Let $(X,d)$ be a metric space, then whenever $a,b\in X$ and $r>0$, $B^d_r(a)=\{x\in X: d(a,x)<r\}$
and  $B^d_r(a,b)=\{(x,y)\in X\times X: d(a,x)<r, d(b,y)<r\}$.
\end{enumerate}
\end{nota}

\begin{Pro}  Assume that $(X,d,d_u)$ is a topometric space, such that $d_u$ $C$-uniformly refines $d$. Assume that $d$ and $d_u$ are bounded by $1$. Then the triple $(L^0([0,1],X),\hat d,\hat d_u)$ is a topometric space and $\hat d_u$ $C$-uniformly refines  $\hat d$.
\end{Pro}

\begin{proof}
First we show that $\hat d_u$ refines the topology induced by $\hat d$. Indeed $\hat d(f,h)=\int_0^1 d(f(\omega),h(\omega)) d\omega \leq \int_0^1 Cd_u(f(\omega),h(\omega)) d\omega\leq C\hat d_u(f,h)$. This shows that $\hat d_u$ $C$-uniformly refines the metric $\hat d$.

 In order to prove that $\hat d_u$ is lower-semicontinuous,  it suffices to prove that whenever $r>0$, the set $V=\{(f,h)\in L^0\times L^0: \hat d_u(f,h)>r\}$ is open with respect to $\hat d$. Let $(f,h)\in V$.
Then $\hat d_u(f,h)>r+\epsilon$ for some $\epsilon>0$.

For each $m\geq 1$ and $0\leq i< m$, let  $A_{i/m,(i+1)/m}=\{\omega\in [0,1]:
i/m< d_u(f(\omega),h(\omega))\leq (i+1)/m\}$.
Since $d_u$ is lower semicontinuous,
the sets $A_{i/m,(i+1)/m}$ are measurable. By the monotone convergence Theorem, 
so by choosing $m$ large enough, we have 

\begin{align}
\sum_{i>0} \frac{i}{m}\mu(A_{i/m,(i+1)/m})>r+\epsilon.
\end{align}

Since $d_u$ is lower semicontinuous, for each $\omega \in A_{i/m,(i+1)/m}$ there is 
$\delta_\omega>0$ such that for all $1/k\leq \delta_\omega$, if $(x_1,x_2)\in B^{d}_{1/k}(f(\omega),h(\omega))$, then $d_u(x_1,x_2)>i/m$.

For each $k>0$, let 
\[\begin{array}{rl}
	A_{i/m,(i+1)/m,k}=&\{\omega \in A_{i/m,(i+1)/m}: \mbox{if } (x_1,x_2)\in B^{d}_{1/k}(f(\omega),h(\omega)),\\
	 & \mbox{ then } d_u(x_1,x_2)>i/m\}.
\end{array}\]%A_{i/m,(i+1)/m,k}=\{\omega \in A_{i/m,(i+1)/m}: \mbox{if } (x_1,x_2)\in B^{d}_{1/k}(f(\omega),h(\omega)), \mbox{ then } d_u(x_1,x_2)>i/m\}.\]

Then there is $N>0$ such that $\mu(A_{i/m,(i+1)/m,N})\geq \mu(A_{i/m,(i+1)/m})-\frac{\epsilon}{4m}$
for all $i< m$.
Let $C=(\cup_{1\leq i<m} A_{i/m,(i+1)/m}) \setminus \cup_{1\leq i<m} A_{i/m,(i+1)/m,N}$, then 
$\mu(C)<\epsilon/4$.

\textbf{Claim} The ${\hat d}$-open ball of center $(f,h)$ and radius ${\epsilon/(4N)}$ is contained in $V$.

To check this, let $(f',h')\in B^{\hat d}_{\epsilon/(4N)}{(f,h)}$ and let 
\[B=\{\omega\in [0,1]:d(f(\omega),f'(\omega))\geq 1/N \mbox{ or } d(h(\omega),h'(\omega))\geq 1/N\}.\] Then 
$\mu(B) < \epsilon/4$. Now let $A=[0,1]\setminus (B\cup C)$, then 
$\mu (A^c)\leq \mu(B)+\mu(C)\leq \epsilon/2$. 

Let $\omega\in A$ and assume that $d(f(\omega),h(\omega))>0$. Then there is some $i$ such that  $\omega \in A_{i/m,(i+1)/m}$. Then \[d(f'(\omega),f(\omega))<1/N,\] \[d(h'(\omega),h(\omega))<1/N,\] and 
\[\omega\in A_{i/m,(i+1)/m,N},\] so $d(f'(\omega),h'(\omega))>i/m$.

From this and equation (1) we get
\[\begin{split}
	&\hat d_u(f',h')=\int d_u(f'(\omega),h'(\omega)) d\omega\geq 
\int_A d_u(f'(\omega),h'(\omega)) d\omega\geq \\
& \sum_{i>0} \frac{i}{m}\mu(A\cap A_{i/m,(i+1)/m,N})
\geq [\sum_{i>0} \frac{i}{m}\mu(A_{i/m,(i+1)/m})]-\mu(A^c)\geq r+\epsilon/2
\end{split}\]
%
%\[\hat d_u(f',h')=\int d_u(f'(\omega),g'(\omega)) d\omega\geq 
%\int_A d_u(f'(\omega),g'(\omega)) d\omega\geq \\ \sum_{i>0} \frac{i}{m}\mu(A\cap A_{i/m,(i+1)/m1,1/N})
%\geq [\sum_{i>0} \frac{i}{m}\mu(A_{i/m,(i+1)/m})]-\mu(A^c)\geq r+\epsilon/2\]. 
So $\hat d_u(f',g')>r$ and $V$ is $\hat d$-open as we wanted.
\end{proof}

\begin{Cor} Let $(G,d,d_u)$ be a Polish topometric group and assume $\hat d_u$ $C$-uniformly refines  $\hat d$. Also assume that $d$ and $d_u$ are bounded by one. Then $(L^0([0,1],G),\hat d,\hat d_u)$ is a Polish topometric group.
\end{Cor}

\begin{proof}
We just proved that $(L^0([0,1],G),\hat d,\hat d_u)$ is a topometric space. Since $d_u$ is bi-invariant it is easy to check that $\hat d_u$ is also bi-invariant. Finally since $(G,d)$ is Polish, then $L^0([0,1],G)$
is also Polish, see \cite{Ke, Ib}.
\end{proof}

We recall the following definition from \cite{BBM}.

\begin{defi}
Let $(G,\tau,d_u)$ be a Polish topometric group.
\begin{enumerate}
	\item We say that $(G,\tau,d_u)$ has \textbf{metric generics} if there is $x\in G$ such that $\overline{Orb(x)}^{d_u}$ is comeager, where the closure is taken with respect to the metric $d_u$ and the orbit with respect to the conjugacy action. We call such an $x$ a \textbf{metric generic element}.
  \item We say that $(G,\tau,d_u)$ has \textbf{ample metric generics} if for every $n$ there is $x\in G^n$ such that $\overline{Orb(x)}^{d_u}$ is comeager,
where the closure is taken with respect to the metric $d_u$ and the orbit with respect to the diagonal conjugacy action. We call such an $x$ a \textbf{metric generic $n$-tuple}.
\end{enumerate}
\end{defi}

\begin{thm}\label{transfmetL} Let $(G,d,d_u)$ be a Polish topometric group with metric generics and assume that $d$, $d_u$ are bounded by one. Then $(L^0([0,1],G),\hat d,\hat d_u)$ has metric generics. Furthermore, if $g\in G$ is a metric generic in $(G,d,d_u)$, then
$C_g$ is a metric generic in $(L^0([0,1],G),\hat d,\hat d_u)$, where $C_g(\omega)=g$ for all $\omega\in [0,1]$.
\end{thm}

\begin{proof}
Since $(G,d,d_u)$ has metric generics, there is $g\in G$ such that $S=\overline{Orb(g)}^{d_u}$ is comeager in $G$ (here we take the closure with respect to $d_u$).  Then by \cite{KL} $Orb(C_g)$ is dense.  Let $F=\{f\in L^0([0,1],G): f\in S \ a.e.\}$. Since $S$ 
contains a dense $G_\delta$ set in $G$, by \cite{KL} $F$ contains a dense $G_\delta$ set in $L^0([0,1],G)$.
It remains to show that $F\subseteq \overline{Orb(C_g)}^{\hat d_u}$ (here we take the closure with respect to $\hat d_u$).

Let $f\in F$ and let $\epsilon>0$. Since for a.e. $\omega\in [0,1]$, $f(\omega)\in S$, there is $h_\omega\in G$ such that $d_u(f(\omega),h_{\omega}^{-1}gh_{\omega})\leq \epsilon$. Since 
$d_u$ is lower-semicontinuous, this is a Borel condition. By Jankov-von Neumann we can find $h \in L^0([0,1],G)$ such that
$d_u(f(\omega),h(\omega)^{-1}gh(\omega))\leq \epsilon$ for a.e. $\omega\in [0,1]$ and thus
$\hat d_u(f,h^{-1}gh)\leq \epsilon$. This shows that $f\in \overline{Orb(C_g)}^{\hat d_u}$ as we wanted.
\end{proof}

\begin{thm}  Let $(G,d,d_u)$ be a Polish topometric group with topometric ample generics, then $(L^0([0,1],G),\hat d,\hat d_u)$ has topometric ample generics.
\end{thm}

\begin{proof}
Let $n\geq 1$ and let $\vec g=(g_1,\dots,g_n)\in G^n$ be such that $S=\overline{Orb(\vec g)}^{d_u}$ is comeager. Consider $C_{\vec g}=(C_{g_1},\dots,C_{g_n})$ the constant function with value $(g_1,\dots,g_n)$. As in the previous proof, it is easy to prove that $\overline{Orb(C_{\vec g})}^{\hat d_u}$ is comeager in $(L^0([0,1],G),\hat d)$.
\end{proof}

\section{Topologies of the group $\mbox{Isom}(L^0([0,1],X)$)}\label{uct}

Let $X$ be a Polish space. Fix $d$ a metric bounded by $1$ that generates the topology and such that $X$ is complete with respect to $d$. Let $G=\mbox{Isom}(X)$, the group of isometries of $(X,d)$. The topology of pointwise convergence on $G$ is Polish. 

For $\alpha,\beta \in L^0([0,1],X)$ let $\hat d(\alpha,\beta)=\int d(\alpha(\omega),\beta(\omega)) d\omega$. This generates a Polish topology on $L^0([0,1],X)$. In \cite{Ib} Ibarluc\'ia studies the group of isometries of $(L^0([0,1],X),\hat d)$.
In particular he characterizes the group of isometries as $\tilde{G}=L^0([0,1],G)\rtimes \Aut$, where the action is given as follows: for $(f,T)\in L^0([0,1],G)\rtimes \Aut$ and $\alpha\in L^0([0,1],X)$, then $((f,T)(\alpha))(\omega)=f(\omega)(\alpha(T^{-1}(\omega)))$. Note that since $G$ is Polish, so is $L^0([0,1],G)$ and that $\Aut$ is Polish with the topology of weak convergence.

Since $\tilde{G}$ is a group of isometries of a Polish metric space, we can also endow it naturally with two topologies:

\begin{defi}
For $(f,T) \in \tilde{G}$; $\alpha_1,\dots,\alpha_n \in L^0([0,1],X)$ and $\epsilon>0$, let
$$V((f,T))=\{(g,S)\in \tilde{G}: \hat d((f,T)(\alpha_i),(g,R)(\alpha_i))<\epsilon, i\leq n\}.$$
We call the topology generated by these sets, where $\epsilon$ varies over the positive numbers,
$\alpha_1,\dots,\alpha_n$ belong to $L^0([0,1],X)$
and $n$ ranges over the positive natural numbers the \textbf{pointwise convergence topology}. It is Polish, see example to find a complete metric that generates the topology.
\end{defi}

\begin{defi}\label{defnunifconv}
For $(f,T), (g,S)\in \tilde{G}$, the uniform distance is given by:

$$L_u((f,T), (g,S))=\sup_{\alpha\in L^0([0,1],X)}\hat d((f,T)(\alpha),(g,S)(\alpha)).$$
Since the collection of simple functions is dense in $L^0([0,1],\mathbb{N})$ we can also take the supremum above over simple functions. This is the topology of \textbf{uniform convergence} as explained in example \ref{du}.
\end{defi}

The goal of this section is to study and characterize the two topologies in $\tilde{G}$
in terms of topologies of $L^0([0,1],G)$ and $\Aut$.

\subsection{Pointwise convergence topology}

As we said earlier, the spaces $L^0([0,1],G)$, $\Aut$ are both Polish, so the group $\tilde{G}$, being the semidirect product of these two groups, also carries the product topology which is Polish. In \cite{Ib} Ibarluc\'ia shows that the product topology coincides with the topology of pointwise convergence. In this section we will prove again this fact, our 
proof is very soft, we will prove by double containment that the topologies agree. We assume that $X$ is 
not trivial, that is, $|X|\geq 2$. We will write $Id$ for the map from $[0,1]$ to $[0,1]$ defined by 
$Id(\omega)=\omega$.

First, we show that the product topology is finer than the pointwise convergence topology.

\begin{lem}
Let $(f,T)\in \tilde G$, let $\alpha\in L^0([0,1],X)$ be a simple function and $\epsilon>0$. Consider
the subbasic open set \[V=\{(g,R)\in \tilde G: \int d(f(\omega)\alpha(T^{-1}(\omega)),g(\omega)\alpha(R^{-1}(\omega))) d\omega<\epsilon\}\] in the topology of pointwise convergence. 
Then there are open sets $U\subseteq L^0([0,1],G)$, $W\subseteq \Aut$ such that $(f,T)\in U\times W\subseteq V$.
\end{lem}

\begin{proof}
We may write $\alpha=\sum_{i=1}^k a_i\chi_{A_i}$ where $A_1,A_2,\dots,A_k$ is a partition of $[0,1]$
and $a_1,\dots,a_k\in X$. 

First let $W:=\{R\in \Aut: \mu(T(A_i)\triangle R(A_i))<\epsilon/(2k),i\leq k\}$ and let
$U:=\{g\in L^0([0,1],G): \int d (f(\omega)(a_i),g(\omega)(a_i))d\omega<\epsilon/(2k):i\leq k\}$.

Then whenever $g\in U$ and $R\in W$ we have that:
%\begin{equation*}
\[\begin{split}
&\int d(f(\omega)\alpha(T^{-1}(\omega)),g(\omega)\alpha(R^{-1}(\omega))) d\omega\\
\leq&\sum_{i,j\leq k} \int_{T(A_i)\cap R(A_j)} d(f(\omega)(a_i)),g(\omega)(a_j)) d\omega\\
\leq&\sum_{i\leq k} \int_{T(A_i)\cap R(A_i)} d(f(\omega)(a_i)),g(\omega)(a_i)) d\omega\\
&\hspace*{0.75 cm}+\sum_{i,j\leq k,i\neq j} \int_{T(A_i)\cap R(A_j)} d(f(\omega)(a_i)),g(\omega)(a_j)) d\omega \\
\leq&\sum_{i\leq k} \epsilon/(2k)+\sum_{i \leq k} \mu (T(A_i)\triangle R(A_i))\leq \epsilon/2+\sum_{i\leq k} \epsilon/(2k) \leq \epsilon.
\end{split}\]
%\end{equation*}

%$\int d_X(f(\omega)\alpha(T^{-1}(\omega)),g(\omega)\alpha(R^{-1}(\omega))) d\omega\leq$\\
%$\sum_{i,j\leq k} \int_{T(A_i)\cap R(A_j)} d_X(f(\omega)(a_i)),g(\omega)(a_j)) d\omega \leq $\\
%$\sum_{i\leq k} \int_{T(A_i)\cap R(A_i)} d_X(f(\omega)(a_i)),g(\omega)(a_i)) d\omega+$\\
%$\sum_{i,j\leq k,i\neq j} \int_{T(A_i)\cap R(A_j)} d_X(f(\omega)(a_i)),g(\omega)(a_j)) d\omega \leq$\\
%$\sum_{i\leq k} \epsilon/(2k)+\sum_{i \leq k} \mu (T(A_i)\triangle R(A_i))\leq \epsilon/2+\sum_{i\leq k} \epsilon/(2k) \leq \epsilon$.

Finally notice that $f\in U$ and $T\in W$.
\end{proof}

The next lemma is the other direction: the pointwise convergence topology is finer than the product topology. 

\begin{lem}
Let $(f,T)\in \tilde G$, let $\alpha\in L^0([0,1],X)$ be simple, $B\subseteq [0,1]$ measurable and $\epsilon>0$. Consider the open sets $U\subseteq L^0([0,1],G)$ defined by $U=\{g\in L^0([0,1],G):
\int d_X (f(\omega) \alpha(\omega),g(\omega) \alpha(\omega)) d\omega<\epsilon\}$ and 
$W\subseteq \Aut$ given by $\{R\in \Aut: \mu(T(B)\triangle R(B))<\epsilon\}$. Then there are
open sets $V_1,V_2$ in $\tilde{G}$ in the topology of pointwise convergence with $(f,T)\in V_1 \subseteq L^0([0,1],X) \times W$ and $(f,T)\in V_2 \subseteq U \times \Aut$.
\end{lem}

\begin{proof}
Since $X$ has more than one point, we may find $c_1,c_2\in X$ with $d_X(c_1,c_2)=s>0$. Let
$\beta_1=c_1\chi_{B}+c_1\chi_{B^c}$ and let 
$\beta_2=c_1\chi_{B}+c_2\chi_{B^c}$ and consider 
\[V((f,T),\beta_1,\beta_2,\epsilon s/4)=\{(g,R):\int d_X((f,T)(\beta_i),(g,R)(\beta_i))<\epsilon s/4,i\leq 2\}.\]

Then whenever $(g,R)\in V((f,T),\beta_1,\beta_2,\epsilon s/4)$, we have

\begin{equation}
\int_{T(B)\cap R(B)^c}d_X(f(\omega)(c_1),g(\omega)(c_1))d\omega<\epsilon s/4,
\end{equation} and
\begin{equation}
\int_{T(B)\cap R(B)^c}d_X(f(\omega)(c_1),g(\omega)(c_2))d\omega<\epsilon s/4,
\end{equation}
so adding inequalities (1) and (2)
\begin{equation*}
\int_{T(B)\cap R(B)^c}d_X(f(\omega)(c_1),g(\omega)(c_1))+d_X(f(\omega)(c_1),g(\omega)(c_2)) d\omega<\epsilon s/2.
\end{equation*}
This shows, using the triangle inequality, that
\begin{equation*}
\int_{T(B)\cap R(B)^c} s d\omega<\epsilon s/2, \mbox{ so } \mu(T(B)\cap R(B)^c)<\epsilon/2.
\end{equation*}

Similarly, $\mu(T(B)^c\cap R(B))<\epsilon/2$, so $\mu(T(B) \triangle R(B))<\epsilon$ as desired.

For the second part, write $\alpha=\sum_{i=1}^k a_i\chi_{A_i}$ where $A_1,A_2,\dots,A_k$ is a partition of $[0,1]$
and $a_1,\dots,a_k\in X$.  By applying the previous argument, we can find a basic open set $V_4$
with $(f,T)\in V_4$ such that whenever $(g,R)\in V_4$ we have $\mu(T(A_i)\triangle R(A_i))<\epsilon/2k$ for $i\leq k$. Now consider $V_3=\{(g,R): \int d_X((f,T)(\gamma),(g,R)(\gamma)) d\omega<\epsilon/2\}$,
where $\gamma(\omega)=\alpha(T(\omega))$. Notice that $(h,Id)(\alpha)=(h,T)(\gamma)$
for all $h\in L^0([0,1],G)$.

Choose $(g,R)\in V_4\cap V_3$. Then
\begin{alignat*}{1}
 &\int d_X (f(\omega) \alpha(\omega),g(\omega) \alpha(\omega)) d\omega\\
=&\int d_X ((f,Id)\alpha(\omega),(g,Id)\alpha(\omega)) d\omega\\
=&\int d_X ((f,T) \gamma(\omega),(g,T) \gamma(\omega)) d\omega\\ 
\leq& \sum_{i\leq k}\int_{A_i\cap R(T^{-1}(A_i))} d_X (f(\omega) \gamma(T^{-1}\omega),g(\omega) \gamma(R^{-1}(\omega)) d\omega+\sum_{i\leq k} \mu(T(A_i)\triangle R(A_i)) \\
\leq& \epsilon/2+\epsilon/2\leq \epsilon.
\end{alignat*}

\end{proof}

Thus, we proved that the product topology is indeed the pointwise convergence topology. 

\subsection{Uniform convergence topology}

In this section we characterize the metric for uniform convergence in terms of 
the metrics for uniform convergence for $\Aut$ and for $L^0([0,1],G)$. By the metric of
uniform convergence in $\Aut$ we mean $\Delta_u(T,R)=\mu\{\omega\in [0,1]: T(\omega)\neq R(\omega)\}$. Note that we 
could also use $\Delta_u'(T,R)=\sup\{\mu(T(A)\triangle R(A)):A\subseteq [0,1]$ measurable $\}$ (see \cite{Ha}).
Similarly, since $G=Isom(X,d)$ has a metric of uniform convergence $d_{u}$, we have by section 
\ref{KlM3} that $(L^0([0,1],G),\tau,\hat{d_{u}})$ is topometric where $\tau$ is the topology of convergence in 
measure, so we can consider $\hat{d_{u}}$ as a natural uniform metric for $L^0([0,1],G)$.

For clarity, we first do the argument for $X=\mathbb{N}$ (where $d(n,m)=1$ if $n\neq m$) and 
$G=S_\infty$ and then we consider the general case. Note that for $\sigma,\rho\in S_\infty$
distinct $d_u(\sigma,\rho)=1$. In what follows, we write $e$ for the identity in $S_\infty$, $C_e$ for the function from $[0,1]\to S_\infty$ with constant value $e$ and $Id$ for the map from $[0,1]$ to $[0,1]$ defined by $Id(\omega)=\omega$.

\begin{Pro}\label{reltopometricSinfty}
For $(f,T) \in L^0([0,1],S_\infty)\rtimes \Aut$, we have:
$$L_u((f,T), (C_e,Id))=\mu(\{\omega \in [0,1]: f(\omega)\neq e\}\cup\{ \omega \in [0,1]: T^{-1}(\omega)\neq \omega\}).$$
\end{Pro}

\begin{proof}

Let $C=\{\omega \in [0,1]: f(\omega)=e$ and $T^{-1}(\omega)=\omega\}$.
Then for any function $\alpha \in L^0([0,1],\mathbb{N})$ and for any $\omega\in C$ we have
\[(f,T)(\alpha)(\omega)=f(\omega)(\alpha(T^{-1}(\omega)))=f(\omega)(\alpha(\omega))=
(\alpha)(\omega)\]
 and thus $d_u((f,T), (C_e,Id))\leq \mu(C^c)$.

For the other inequality, let us define 
$A=\{\omega \in [0,1]: f(\omega)\neq e \wedge T^{-1}(\omega)=\omega\}$.

For each $\omega\in A$ let $n_\omega$ be the minimum natural number such that 
$f(\omega)(n_\omega)\neq n_\omega$ and let 
$\alpha(\omega)=n_\omega$. Note that $\alpha$ has been defined on $A$ and in 
this set it is measurable. Also note that for $\omega \in A$,
$(f,T)(\alpha)(\omega)=f(\omega)(\alpha(T^{-1}(\omega)))=f(\omega)(n_\omega)\neq n_\omega=\alpha(\omega)$ so the two automorphisms $(f,T)$, $(C_e,Id)$ disagree in every 
$\omega \in A$ when they act in $\alpha$.

Now let $B=\{\omega \in [0,1]: T^{-1}(\omega)\neq \omega\}$.
We may write $B=B_0\cup \cup_{i\geq 2} B_i$ where $B_0$ are the points where 
$T$ is an aperiodic map and $B_i$ are the points where $T$ is a cycle 
of period $i$. All the sets $B_i$ are measurable.

First we will deal with $B_2=\{\omega\in B: T^2(\omega)=\omega\}$. Let $C_2$ be a 
measurable subset of $B_2$ such that $C_2$, $T(C_2)$ are disjoint and 
$B_2=C_2 \cup T(C_2)$.

%For $\omega\in C_2$, let $\alpha_1(\omega)=0$ and $\alpha_1(T(\omega))=1$.

%Note that for $\omega\in C_2$, as long as $f(\omega)(1)\neq 0$ we have that
%$f(\omega)(\alpha_1(T^{-1}(\omega)))=f(\omega)(1)\neq 0$ but $\alpha_1(\omega)=0$.
%Similarly, for $\omega\in T(C_2)$, as long as $f(\omega)(0)\neq 1$ we have that
%$f(\omega)(\alpha_1(T^{-1}(\omega)))=f(\omega)(0)\neq 1$ but $\alpha_1(\omega)=1$.
%So $\alpha_1$ witnesses that $(f,T)\alpha_1(\omega)\neq (C_e,Id)\alpha_1(\omega)$ for all
%$\omega\in B_2$ outside the set $\{\omega\in C_2:g(\omega)^{-1}(0)= 1\}\cup \{\omega\in %T(C_2):g(\omega)(0)= 1\}$

For $n\geq 1$, define $\alpha_n$ as follows. For $\omega\in C_2$, let $\alpha_n(\omega)=0$ and $\alpha_n(T(\omega))=n$. It is easy to observe that $\alpha_n$ satisfies $(f,T)\alpha_n(\omega)\neq (C_e,Id)\alpha_n(\omega)$ for $\omega\in B_2$ outside the set $\{\omega\in C_2:f(\omega)^{-1}(0)= n\}\cup \{\omega\in T(C_2):f(\omega)(0)= n\}$.
Since $f$ is fixed, $\mu(\{\omega\in C_2:f(\omega)^{-1}(0)= n\})\to 0$ as $n\to \infty$. Similarly,
$\mu(\{\omega\in T(C_2):f(\omega)(0)= n\})\to 0$ as $n\to \infty$. 

Thus $\lim_{n\to \infty} \int_{B_2} \hat d ((f,T)(\alpha_n),(C_e,Id)(\alpha_n))d\omega=\mu(B_2)$,
so restricting to $B_2$ we get $d_u((f,T),(C_e,Id) )\geq \mu(B_2)$.

We deal with $B_3=\{\omega\in B: T^3(\omega)=\omega\}$ in a similar way. Let $C_3$ be a 
measurable subset of $B_3$ such that $C_3$, $T(C_3)$, $T^2(C_3)$ are pairwise disjoint and 
$B_3=C_3 \cup T(C_3)\cup T^2(C_3)$.
For $\omega\in C_3$ and $n\geq 1$, let $\beta_n(\omega)=0$ and $\beta_n(T(\omega))=n$, 
$\beta_n(T^2(\omega))=2n$.

As above, $\beta_n$ separates $(f,T)$ from $(C_e,Id)$ in $B_3$ outside the set of exceptional points
$\{\omega\in C_3:f(\omega)^{-1}(0)= n\}\cup \{\omega\in T(C_3):f^{-1}(\omega)(n)= 2n\}\cup 
\{\omega\in T^2(C_3):f(\omega)(0)= 2n\}$. Notice that the measure of the set of exceptional points 
goes to zero as $n$ goes to infinty, so $\lim_{n\to \infty} \int_{B_3} \hat d ((f,T)(\beta_n),(C_e,Id)(\beta_n))d\omega=\mu(B_3)$ .

Using a similar approach for all the periodic pieces and approximating the aperiodic piece by periodic pieces using Rohklin's Lemma we get the desired result.

\end{proof}

\begin{Cor}\label{prodtopunif}
For $(f,T) \in L^0([0,1],S_\infty)\rtimes \Aut$, we have:\\
$\max\{\int d_u(f(\omega),e) d\mu,\Delta_u(T,Id)\}\leq 
L_u((f,T), (C_e,Id))\leq
\int d_u(f(\omega),e) d\mu +\Delta_u(T,Id)$

this shows that the metric of uniform continuity is uniformly equivalent to the product distance of the metrics $\hat d_u $ (in $L^0([0,1],S_\infty)$) and $\Delta_u$ (in $\Aut$).
\end{Cor}

Now we consider the general case, so $(X,d)$ is a Polish metric space with $diam(X)\leq 1$ and $G=\mbox{Isom}(X,d)$.
We also assume the set $X$ has more than one point, so we may find $a,b\in X$ with 
$d(a,b)=r>0$.

\begin{Pro}\label{reltopometric}
For $(f,T), (h,S)\in L^0([0,1],G)\rtimes \Aut$, the functions sending $\omega\in [0,1]$ to
$d_u(f(T^{-1}(\omega)), h(S^{-1}(\omega)))$ and $d_G(f(T^{-1}(\omega)), h(S^{-1}(\omega)))$ are both measurable and integrable.
\end{Pro}

\begin{proof}

First observe that since $f$ is measurable and $T$ is an invertible  measurable preserving transformation, the map that sends $\omega$ to $f(T^{-1}(\omega))$ is measurable. So is the map that 
sends $\omega$ to $h(S^{-1}(\omega))$. Finally, since the map $d_u:G\times G\to [0,1]$ is Borel measurable, then the map from $[0,1]$ to $[0,1]$ that sends $\omega$ to $d_u(f(T^{-1}(\omega)),h(S^{-1}(\omega)))$ is measurable.

Similarly, since the function $d:G\times G\to [0,1]$ is Borel measurable, the map from $[0,1]$ to $[0,1]$ that sends $\omega$ to $d_G(f(T^{-1}(\omega)),h(S^{-1}(\omega)))$ is measurable.

Since both functions are bounded by $1$, they are also integrable.
\end{proof}

\begin{thm}\label{unifmetric} Let
\[A=\{\omega \in [0,1]: f(\omega))\neq \omega \wedge T^{-1}(\omega)=\omega\},\] 
\[B=\{\omega \in [0,1]: T(\omega)\neq \omega\}.\]
Then
\[\frac{r}{8} \mu(B)+\int_A d_u(f(\omega), e)d\mu \leq L_u((f,T), (C_e,Id))\leq \\ \mu(B)+\int_A d_u(f(\omega), e)d\mu.\]
\end{thm}

\begin{proof}
Let $\alpha\in L^0([0,1],X)$ and let $B=\{\omega \in [0,1]: T(\omega)\neq \omega\}$. Then for each $\omega\in B$, $d(f(\omega)(\alpha(T^{-1}(\omega)), \alpha(\omega)))\leq 1$. 
On the other hand for any $\omega\in A$, $d_u(f(\omega), e)\geq d(f(\omega)(\alpha(\omega)), \alpha(\omega))$, 
so $\int_A d_u(f(\omega), e)d\mu \geq \int_A d(f(\omega)(\alpha(\omega)), \alpha(\omega)) d\mu$.
This shows that $L_u((f,T), (C_e,Id))\leq \mu(B)+\int_A d_u(f(\omega), e)d\mu$.

We now prove the other inequality. Let $\epsilon>0$. 

For each $\omega\in A$ let $\alpha_\omega\in X$ be  such that 
$d(f(\omega)(\alpha_\omega),\alpha_\omega))+\epsilon\geq d_u(f(\omega),e)$. Since $d$, $d_u$ are measurable, we may define a measurable function $\alpha$ on $A$ so that
$d(f(\omega)(\alpha(\omega)),\alpha(\omega))+\epsilon\geq d_u(f(\omega),e)$. 
Note that for $\omega \in A$, $\int_A d(f(\omega)(\alpha(\omega)),\alpha(\omega))d\omega +\epsilon\geq \int_A d_u(f(\omega),e)d\omega$. 

Recall that $B=\{\omega \in [0,1]: T^{-1}(\omega)\neq \omega)\}$ and as before, write $B=B_0\cup \cup_{i\geq 2} B_i$ where $B_0$ are the points where $T$ is an aperiodic map and $B_i$ are the points where $T$ is a cycle 
of period $i$. All the sets $B_i$ are measurable and $T(B_n)=B_n$ for all $n$.

For each $n\geq 2$ we define a function in 
$B_n$. Consider the sets $D_n=\{\omega\in B_n: d(a,f(\omega)(a))\geq r/2\}$ and 
$E_n=\{\omega\in B_n: d(a,f(\omega)(b))\geq r/2\}$. By the triangle inequality, one of the two sets
must have measure $\geq \mu(B_n)/2$. 

\textbf{Case 1} $\mu(D_n)\geq \mu(B_n)/2$.

Define $\alpha_1=a\chi_{B_n}$, then whenever $\omega\in D_n$, we have 
$d((f,T)(\alpha_1)(\omega),(C_e,Id)(\alpha_1)(\omega))=d(f(\omega)(a),a)\geq r/2$, 
so $\int_{B_n}d((f,T)(\alpha_1)(\omega),(C_e,Id)(\alpha_1)(\omega)) d\omega\geq \mu(D_n)r/2
\geq \mu(B_n)r/4$.

\textbf{Case 2} $\mu(E_n)\geq \mu(B_n)/2$.

Assume for the sake of simplicity $n=2k$ and that there is a measurable set
 $C_n$ such that $C_n,\dots, T^{2k-1}(C_n)$ are pairwise disjoint, form a partition of $B_n$
and all of them are independent from $E_n$ (this last part can be assured up to $\delta$
for any $\delta>0$). Define $\alpha_2=a\chi_{C_n\cup T^2(C_n)\cup \dots \cup T^{2k-2}(C_n)}+
b\chi_{T(C_n)\cup \dots \cup T^{2k-1}(C_n) }$. Then whenever 
$\omega\in E_n\cap (C_n\cup T^{2}(C_n)\cup \dots \cup T^{2k-2}(C_n))$, we have that
$d((f,T)(\alpha_2)(\omega),(C_e,Id)(\alpha_2)(\omega))=d(f(\omega)(b),a)\geq r/2$
so $\int_{B_n}d((f,T)(\alpha_2)(\omega),(C_e,Id)(\alpha_2)(\omega)) d\omega\geq 
(\mu(E_n)/2) \cdot (r/2) \geq \mu(B_n)r/8$.

\end{proof}

Note that the above formula implies that  $$\frac{r}{8} \max\{ \mu(B),\hat d_u(f, C_e)\} \leq L_u((f,T), (C_e,Id))\leq \mu(B)+\hat d_u(f, C_e)$$ so the uniform metric corresponds to the product topology of the uniform metrics $\hat d_u$ in $L^0([0,1],G)$ and $\Delta_u$ in $\Aut$.

\section{The Rohklin property in $\tilde{S}_\infty$}\label{Den}

In this section we study the Rohklin property on  $\tilde{S}_\infty$. We notice first that the strong Rohklin property is not satisfied by any $\tilde{G}$, so $\tilde{G}$ will not have ample generics. We then show that $\tilde{S}_\infty$ has the Rohklin property. From this proof, we can extract a sufficient condition for $\tilde{G}$ to have the Rohklin property, namely the Rohklin property under powers defined below. We also show that this is a rather robust definition, by showing several groups that have this property. In the next section, we will show via a topological argument the general case, however, we keep this argument as it shows a different approach (more algebraic in nature) but interesting on itself.
 
\begin{Pro}
Let $G$ be a topological group. Then $\tilde{G}$ does not have the strong Rokhlin property.
\end{Pro}

\begin{proof}
Suppose that $(f,R)\in \tilde{G}$ has a comeager orbit. 
%Notice that the projection on $\Aut$ of $[(f,r)]_{\tilde{G}}$ is exactly $[r]_G$. %Since the former is dense, the latter must also be. 
By the Kuratowski-Ulam Theorem, there is a comeager set $X_G$ such that for every $g\in X_G$, $\{T: (g,T)\in [(f,R)]_{\tilde{G}}\}$ is comeager. So fix such $g$ and notice that if $(g,T)\in [(f,R)]_{\tilde{G}}$, then $T$ and $R$ are conjugates. Therefore the orbit of $R$ is comeager, which we know it is not the case.
\end{proof}

\begin{thm} ${\tilde{S}}_{\infty}$ has the Rohklin property.
\end{thm}

\begin{proof}
Choose $\sigma\in S_\infty$ such that its conjugacy orbit is comeager. Remember that $\sigma$ must have infinite copies of any $k$-cycle for every $k\geq 1$. Denote by $C_\sigma$ the constant function in $\Lcero$ with value $\sigma$. Let $S\in \Aut$ with a dense conjugacy orbit, so $S$ is aperiodic. We claim that $(C_\sigma, S)$ has a dense orbit. 

Let $U\times V$ be an open subset of ${\tilde{S}}_\infty$. Since the orbit of $S$ is dense, there is $T\in\Aut$ such that $T^{-1}ST\in V$. 

Let us also fix a simple function $h(x):=\Sigma_{i<M} \tau_i\chi_{A_i}(x)$ in $TU$. Furthermore, we can suppose that the $A_i$ are Borel sets. Notice there is $\varepsilon>0$ and there is $K\in\mathbb N$ such that 
\[\{f\in\Lcero| \mu(\{x\in [0,1] \ | \ \;\forall n<K\; f(x)(n)= h(x)(n)\})\geq 1-\varepsilon\}\subseteq TU.\]

So we need to find $g\in\Lcero$ such that
\begin{equation} \label{eq:1}
g(x)h(x)(n)=\sigma g(S^{-1}(x))(n),
\end{equation}
for all $n\leq K$ and all $x$ except for a set of measure smaller than $\varepsilon$. % (\textbf{Pregunta: usamos $S$ o $T$ en la f\'ormula?, como las \'orbitas aperiodicas son densas podemos suponer que $T$ es aperi\'odico})

Since $S$ is aperiodic, by Rohklin's Lemma, given $\varepsilon$ as above and any $N\geq 1$ there is a measurable subset $E\subseteq [0,1]$ such that $E,S(E),\dots,S^{N-1}(E)$ are pairwise disjoint and $\mu(\cup_{0\leq i\leq N-1}S^i(E))\geq 1-\epsilon/2$. Choose $N$ such that $1/N<\epsilon/2$. We now define $S_0$ so that it coincides with $S$ on $\cup_{0\leq i\leq N-2}S^i(E)$, as $S^{-N+1}(x)$ for $x\in S^{N-1}(E)$ and as a periodic map of period $N$ in $[0,1] \setminus\cup_{0\leq i\leq N-1}S^i(E)$. Thus the map $S_0$ is a measure preserving transformation with period $N$ such that $d_u(S,S_0)=\mu(\{x\in [0,1]| S(x)\neq S_0(x)\})\leq \epsilon/2+1/N<\epsilon$ (this is called by Halmos the Uniform Approximation Theorem). By enlarging $E$ if necessary and using the fact that $S_0$ is periodic with period $N$, we may find a new set $E_0\supseteq E$ such that $E_0,S_0(E_0),\dots,S_0^{N-1}(E_0)$ are pairwise disjoint and $\mu(\cup_{0\leq i\leq N-1}S_0^i(E_0))=1$.

As said earlier, we need to find $g\in\Lcero$ that satisfies equation \ref{eq:1} % such that $g(x)(h(x)(n))=\sigma(g(S^{-1}(x)))(n)$
for all $x$ but a set of measure smaller than $\varepsilon$ and all $n\leq K$; thus it suffices to find $g\in\Lcero$ such that 
\begin{equation} \label{eq:2}
g(x)h(x)(n)=\sigma g(S_0^{-1}(x))(n)
\end{equation}
for all $x$ and all $n<K$. 

Let us see what we need. Assume we have defined a function $g$ that satisfies (\ref{eq:2}) for all $n<K$. We fix $x\in [0,1]$. If we consider all the conditions that we can derive for each $n\leq K$ and its orbit under $S_0$, we obtain the following diagram, which must ``commute'' (over $n\leq K$):

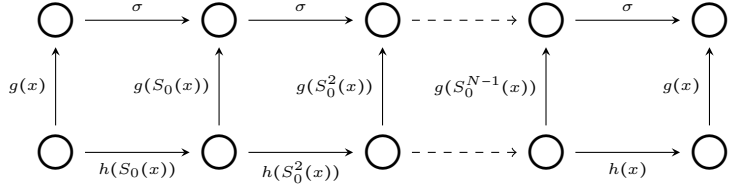
\newcommand\Bigcircle{\raisebox{-0.5mm}{\scalebox{1.7}{$\bigcircle$}}}

\begin{tikzpicture}
  \matrix (m) [matrix of math nodes,row sep=3em,column sep=4em,minimum width=2em]
  {
    \Bigcircle & \Bigcircle &\Bigcircle &\Bigcircle &\Bigcircle \\
    \Bigcircle & \Bigcircle &\Bigcircle &\Bigcircle &\Bigcircle  \\};
  \path[-stealth]
    (m-1-1) edge node [above] {\tiny$\sigma$} (m-1-2)
		(m-1-2) edge node [above] {\tiny$\sigma$} (m-1-3)
		(m-1-3) edge [dashed,->] (m-1-4)								
		(m-1-4) edge node [above] {\tiny$\sigma$} (m-1-5)	
		(m-2-5) edge node [left] {\tiny$g(x)$} (m-1-5)
		(m-2-1) edge node [below] {\tiny$h(S_0(x))$} (m-2-2)
						edge node [left] {\tiny$g(x)$} (m-1-1)
		(m-2-2) edge node [below] {\tiny$h(S_0^2(x))$} (m-2-3)
						edge node [left] {\tiny$g(S_0(x))$} (m-1-2)
		(m-2-3) edge [dashed,->] (m-2-4)
						edge node [left] {\tiny$g(S_0^2(x))$} (m-1-3)
		(m-2-4) edge node [below] {\tiny$h(x)$} (m-2-5)
						edge node [left] {\tiny$g(S_0^{N-1}(x))$} (m-1-4);
\end{tikzpicture}

In other words, the diagram represents, for each $n\leq K$ the different equations that all of its corresponding images must satisfy. Notice that for each $n$, the corresponding images are going to be finite, so that we can also think of the circles from the previous diagram representing the union of all the images where we need the equations to be satisfied. 
%
%\textbf{\begin{tikzpicture}[circ/.style={draw,circle, inner sep=0pt ,minimum size=1em]}]
%\matrix (m) [ampersand replacement=\&, row sep=3em,column sep=4em,minimum width=2em]
  %{
   %\node[circ, anchor=east]{}; \& \node[circ]{}; \& \node[circ]{};\& \node[circ]{};\& \node[circ]{}; \\
   %\node[circ, anchor=east]{};\& \node[circ]{};\& \node[circ]{};\& \node[circ]{};\& \node[circ]{}; \\};
  %\path[-stealth]
    %(m-1-1) edge node [left] {\tiny$g(x)$} (m-2-1)
     %edge node [above] {\tiny$\sigma$} (m-1-2)
		%(m-1-2) edge node [left] {\tiny$g(S_0(x))$} (m-2-2)
            %edge node [above] {\tiny$\sigma$} (m-1-3)
		%(m-1-3) edge node [left] {\tiny$g(S_0^2(x))$} (m-2-3)
            %edge [dashed,->] (m-1-4)								
		%(m-1-4) edge node [left] {\tiny$g(S_0^{N-1}(x))$} (m-2-4)
            %edge node [above] {\tiny$\sigma$} (m-1-5)	
		%(m-1-5) edge node [left] {\tiny$g(x)$} (m-2-5)
		%(m-2-1) edge node [below] {\tiny$h(S_0(x))$} (m-2-2)
		%(m-2-2) edge node [below] {\tiny$h(S_0^2(x))$} (m-2-3)
		%(m-2-3) edge [dashed,->] (m-2-4)
		%(m-2-4) edge node [below] {\tiny$h(x)$} (m-2-5);
%\end{tikzpicture}}

Then we get $g(S_0(x))h(S_0(x))(n)=\sigma(g(x)(n))$ and $g(S_0^2(x))(h(S_0^2(x)(n))=\sigma g((S_0(x))(n)$ and thus (putting together the first two pieces of the diagram) $g(S_0^2(x))h(S_0^2(x))h(S_0(x))(n)=\sigma^2(g(x)(n))$. Applying the same argument $N-1$ times we get 
\begin{equation}
\label{eq:3}
g(x)h(S_0^{N-1}(x))\cdots h(S_0(x))h(x)(n) =\sigma^Ng(x)(n).
\end{equation}

%\begin{equation}
%\label{eq:3}\begin{split}
%g(x)h(S_0^{N-1}(x))\cdots h(S_0(x))(n) &=g(S_0^N(x))h(S_0^{N-1}(x))\cdots h(S_0(x))(n)\\&=\sigma^Ng(x)(n)
%\end{split}
%\end{equation}

We build the desired function $g(x)$ backwards, starting from equation (\ref{eq:3}). %$$g(x)h(S_0^{N-1}(x))\cdots h(S_0(x))(n)=\sigma^N(g(x)(n))$$
So fix $x\in E_0$ and define \begin{equation}
\label{eq:4}f(x):=h(S_0^{N-1}(x))h(S_0^{N-2}(x))\ldots h(S_0(x)) h(x).
\end{equation} 

Since $\sigma^N$ is a generic element in $S_{\infty}$, for each $x\in E_0$ we can find a permutation $\rho_x\in S_\infty$ such that $\rho_x^{-1}\sigma^N\rho_x(n)=f(x)(n)$ for all $n\leq K$. Thus, whenever $x\in E_0$, define $g(x)$ to be one of such permutations $\rho_x$.

Now suppose that $g$ is defined in $E_0$ for all $n<K$ and satisfies the above property. We extend $g$ to $\cup_{1\leq i\leq N-1}S_0^i(E_0)$ so that on $S_0^i(x)\in S_0^i(E_0)$ it satisfies the following equation:
\[g(S_0^i(x))(h(S_0^{i-1}(x))\ldots h(S_0(x)) h(x)(n))=\sigma^ig(x)(n): \ \  n\leq K\]

This might not define $g(S_0^i(x))$ for all $n<K$. But for those values not defined by this equation, we use the same argument as in $E_0$. This is well defined, since $S_0$ is periodic, so that no value will be repeated if it was chosen before. An easy verification shows us that $g(x)(h(x)(n))=\sigma(g(S_0^{-1}(x)))(n)$ for all $x\in E_0\cup S_0(E_0)\cup \ldots S_0^{N-1}(E_0)$ as desired. 

This process does not necesarilly defines $g(x)(n)$ for all $n$ and as constructed the function $g$ need not be measurable. However, this can be solved using Jankov-von Neumann uniformization, as in \cite{KL}. 
Indeed, let us now show that we can find such a $g\in \Lcero$. Let $P\subseteq E_0\times S_\infty$ be defined by
$$(x,\rho)\in P \iff \rho^{-1}\sigma^N\rho(n)=f(x)(n): \ \ x\in E_0,n<K.$$
By Jankov-von Neuman uniformization, we can choose for each $x\in E_0$ a $g_0(x)\in S_\infty$ in 
a measurable way. 

Now let $P_1\subseteq S_0(E_0)\times S_\infty$ be defined by $(x,\rho)\in P_1$ if and only if:

$$ \rho h(S_0(x)) h(x)(n)=\sigma g_0(S_0^{-i}(x))(n): n<K, x\in S^i(E_0).$$

Recall that $h(x)$ is a simple measurable function and since $S_0$ is Borel measurable, so is $h(S_0(x))$. This proves that $h(S_0(x))h(x)$ is also a simple measurable function. Since we have a finite intersection of Borel conditions, $P_1$ is a Borel set. Again by Jankov-von Neuman uniformization, we can choose for each $x\in S_0(E_0)$ a $g_0(x)\in S_\infty$ in a measurable way. 

%%\textbf{Necesitamos que $E_0$ sea Borel y que $S_0$ sea Borel medible?}

Repeating this process for all $i$, we can choose $g(x)\in S_\infty$ in a measurable way.

\end{proof}

Finally we see how to generalize the above result to a wide class of Polish groups. 

\begin{defi}
Let $G$ be a Polish group. We say that $G$ has the \emph{Rohklin property under powers} if there is $g\in G$
such that for all $N\geq 1$, the orbit of $g^N$ under conjugation is dense.
\end{defi}

\begin{thm} If $G$ is a Polish group that has the \emph{Rohklin property under powers}, then ${\tilde{G}}$ has the Rohklin property.
\end{thm}

Instead of doing the proof again in the general setting, we point out how to modify the previous proof.

\begin{proof}
We choose a right-invariant complete metric on $G$, and call it $d_G$.

Choose $\sigma\in G$ that witnesses the Rohklin property under powers, and $S\in Aut([0,1])$ aperiodic. As in the previous proof, we will show that the tuple $(C_\sigma,S)$ has a dense orbit. Let $U\times V$ be open in $\tilde{G}$. Once $T$ has been chosen, we can choose $h\in TU$, and choose $\varepsilon>0$ such that

\[\{f\in \Lcero| \mu(\{x\in [0,1]| d_G(f(x),h(x))<\varepsilon\})\geq 1-\varepsilon\}\subseteq TU.\]

Notice that the choices of $S_0$ and $E_0$ in the proof for $S_\infty$ do not depend at all on $\Lcero$. 
Thus we can choose them in the same way. As before, the idea is to define $g(x)$ for $x$ in $E_0$, and then just expand it to all of the interval. 

However, we no longer have an initial segment of the positive integers to do so. Instead, we want the following inequality to hold for all $x$ in $\cup_{i<n}S^n_0(E_0)$.

\[d_G(g(x)^{-1}\sigma g(S_0^{-1}(x)),h(x))<\varepsilon.\]

With the same idea in mind as in the other proof, we define the Borel relation $P\subseteq E_0\times G$ by:
\[(x,\rho)\in P \Leftrightarrow d(\rho^{-1}\sigma^n\rho, h(x)h(S_0^{N-1}(x))\ldots h(S_0(x)))\leq \varepsilon.\]

Notice this set is not empty, since $G$ has the Rohklin property under powers. Thus, it has a measurable uniformization. We define $g:E_0\to G$ to be this uniformization. We expand it to all of $\cup_{i<n}S^n_0(E_0)$ by the following equation.
\[g(S_0^i(x))=\sigma^ig(x)h(S_0(x))^{-1}h(S_0^2(x))^{-1}\ldots h(S_0^{i}(x))^{-1}.\]

We claim that $g$ so defined satisfies that for all $x\in \cup_{i<n}S^n_0(E_0)$ 

\[d_G(g(x)^{-1}\sigma g(S_0^{-1}(x)),h(x))\leq \epsilon.\]

The calculation is straightforward, but as an illustration of what is happening, let us show it for $x\in E_0$. 

\[\begin{array}{rl}
	d(g(x)^{-1}\sigma g(S_0^{-1}(x)),h(x))=& d(g(x)^{-1}\sigma g(S_0^{N-1}(x),)h(x))\\
	=&d(g(x)^{-1}\sigma g(x)\sigma^{N-1}h(S_0(x))^{-1}\ldots h(S_0^{N-1}(x))^{-1},h(x)) \\
	=&d(g(x)^{-1}\sigma^{N}g(x),h(x) h(S_0^{N-1}(x))\ldots h(S_0(x))) \\
	\leq& \varepsilon
\end{array}
\]

Since the last claim is valid for all $x\in \cup_{i<n}S^n_0(E_0)$, and this set has measure at least $1-\varepsilon$, 
then $(C_\sigma,T)$ has a dense orbit.
\end{proof}

\begin{Pro} The group $G=\Aut$ has the Rohklin property under powers. In particular, $\tilde{G}$ has the Rohklin property.
\end{Pro}

\begin{proof}
Notice that if $T$ is aperiodic, then so is $T^n$ for any $n\neq0$.
\end{proof}

\begin{Pro} Let $\mathcal{U}$ be the group of unitary transformations on a separable Hilbert space. Then $\mathcal{U}$ has then Rohklin property under powers and $\tilde{\mathcal{U}}$ has the Rohklin property.
\end{Pro}

\begin{proof}
Notice that $T\in \mathcal{U}$ is generic when its spectrum $\sigma(T)=S^1$. Then for any $n\geq 1$,  
$\sigma(T^n)=S^1$, so $T^n$ is also generic.
\end{proof}

\begin{Pro} The group $G=Aut(\mathbb Q,\leq)$ has the Rohklin property under powers. In particular, $\tilde{G}$ has the Rohklin property.
\end{Pro}

\begin{proof} We use the characterization of the generic element found in \cite{T}. Given $g\in G$ and $x\in\mathbb{Q}$, define the orbital of $x$ by $g$ as the following set:
$$\mbox{obt}(x,g):=\{y\in\mathbb{Q}| \exists m,n \in \mathbb{Z}\, g^n(x)\leq y\leq g^m(x)\}.$$
and the sign of this orbital as:
$$\mbox{sgn}(x,g):=\left\{\begin{array}{rl}
	1 &\mbox{if } x< g(x),\\
	-1 &\mbox{if } x> g(x),\\
	0 &\mbox{if } x= g(x).
\end{array}\right.$$

Notice that both $\mbox{obt}(x,g)=\mbox{obt}(x,g^n)$ and $\mbox{sgn}(x,g)=\mbox{sgn}(x,g^n)$ for any $n$. 

In \cite{T}, Truss showed that $g$ is generic if the set of orbitals with sign $\epsilon$ is a dense linear order without endpoints and is dense in the union of the other two. Thus, if $g$ is generic, $g^n$ is generic.

\end{proof}

\section{Transferring the Rohklin property to $\tilde{G}$}\label{Den2}

In this section we prove that the Rohklin property transfers from $G$ to $\tilde{G}$. Instead of the algebraic 
approach from the previous section we follow a topological approach that works for all Polish groups.

\begin{thm} If $G$ has the Rohklin property, then so does $\tilde{G}$.
\end{thm}

\begin{proof}
Recall that the Polish topology in $\tilde{G}$ is the product topology, so it suffices to show that for any
non-empty open sets $U_1, U_2\subseteq L^0([0,1],G)$ and $V_1,V_2\subseteq \Aut$ there is
$(f,T)\in \tilde{G}$ such that $(U_1\times V_1)^{(f,T)}\cap (U_2\times V_2)\neq \emptyset$.

So assume we have such sets. Since $\Aut$ has a dense orbit,
we can find $T\in \Aut$ such that $V_1^T\cap V_2\neq \emptyset$. Since conjugation 
is a homeomorphism of the space $\tilde{G}$, $W_1=(U_1\times V_1)^{(C_e,T)}$ is open in $\tilde{G}$. 

Notice that the projection on the second component of $W_1$ coincides with $V_1^T$ and intersects $V_2$. Thus we can find
open sets $U_3\subseteq L^0([0,1],G)$ and $V_3\subseteq \Aut$ such that $U_3\times V_3\subseteq W_1$
and $V_3\subseteq V_2$. 

Since $G$ has a dense orbit, $L^0([0,1],G)$ also has a dense orbit (see \cite{KL}). Thus we can find $f\in L^0([0,1],G)$ such that $U_3^f\cap U_2\neq \emptyset$. This proves that 
$(U_3\times V_3)^{(f,Id)}\cap (U_2\times V_3)\neq \emptyset$ and so we obtain that
$(U_3\times V_3)^{(f,Id)}\cap (U_2\times V_2)\neq \emptyset$. Likewise,
$(W_1)^{(f,Id)}\cap (U_2\times V_2)\neq \emptyset$ and thus $((U_1\times V_1)^{(C_e,T)})^{(f,Id)}\cap (U_2\times V_2)\neq \emptyset$.
\end{proof}

From the proof, we obtain some corollaries.

\begin{Cor}Let $H,G$ are topological groups such that $H$ acts continuously on $G$. If $H$ and $G$ have the Rohklin property, so does $G\rtimes H$. 
\end{Cor}

\begin{Cor} Consider the action of $G$ on $G^n$ by diagonal conjugation, that is, 
each $g\in G$ sends $(g_1,\dots,g_n)$ to $(g_1^g,\dots,g_n^g)$.  Assume $G^n$ has a dense orbit under diagonal conjugation, then so does $\tilde{G}^n$. 
\end{Cor}

\begin{proof}
The proof is the same as before but now we consider open subsets $U_1, U_2\subseteq L^0([0,1],G)^n$
and $V_1,V_2\subseteq (\Aut)^n$ and use the fact that $\Aut$ has ample metric generics and that 
if $G^n$ has a dense orbit under diagonal conjugation so does $L^0([0,1],G)^n$ (see \cite{KL}).
\end{proof}

Below we will use the following notation. We write $d_u$ for the metric of uniform convergence
for $G$, $\hat d_u$ for the induced metric of uniform convergence in $L^0([0,1],G)$, $L_u$ for the metric
of uniform convergence (see Definition \ref{defnunifconv}) in $\tilde{G}$. For $B\subseteq G$, the set
$\overline{B}^{d_u}$ stands for its closure with 
respect to the metric $d_u$ and for $A\subseteq \tilde{G}$, $\overline{A}^{L_u}$ stands for its closure with 
respect to the metric $L_u$.

\begin{Pro}\label{orbitaperiodic} Let $h \in G$, let $C_h:[0,1]\to G$ be the function with constant value $h$
 and let $T\in \Aut$ be aperiodic. Then $\overline{(C_h,T)^{\tilde{G}}}^{L_u} \supseteq \{(C_h,S): S$ is aperiodic $\}$. 
\end{Pro}

\begin{proof}
Let $S$ be aperiodic and $\epsilon>0$. By Rokhlin's Lemma  there is $R\in \Aut$  such that
$\mu\{\omega\in [0,1]: R^{-1}TR(\omega)\neq S(\omega)\}<\epsilon$. Note that since $C_h$ is a constant function 
the action by $R$ on $C_h$ is trivial. Thus by Theorem \ref{unifmetric}
we get that $L_u((C_h,T)^{(C_e,R)},(C_h,S))=L_u((C_h,R^{-1}TR),(C_h,S))<\epsilon$.
\end{proof}

\begin{Obse}\label{doubleconj} Let $f,h \in L^0([0,1],G)$, let $R,T\in \Aut$ and assume that 
$(f,R)\in \overline{(h,T)^{\tilde{G}}}^{L_u}$. Then whenever 
$(f_1,R_1)\in (f,R)^{\tilde{G}}$ we also have $(f_1,R_1)\in \overline{(h,T)^{\tilde{G}}}^{L_u}$.
\end{Obse}

\begin{proof}
Let $\epsilon>0$ and let $(g,S)\in \tilde{G}$ be such that 
$L_u((f,R),(h,T)^{(g,S)})<\epsilon$. Also let $(f_2,R_2)\in \tilde{G}$ be such that 
$(f_1,R_1)=(f,R)^{(f_2,R_2)}$ Since $L_u$ is biinvariant, 
$L_u((f_1,R_1),((h,T)^{(g,S)})^{(f_2,R_2)})=L_u((f,R)^{(f_2,R_2)},((h,T)^{(g,S)})^{(f_2,R_2)})<\epsilon$
as we wanted.
\end{proof}

\begin{thm} If $(G,\tau,d_u)$ has metric generics, then so does $\tilde{G}$.
\end{thm}

\begin{proof}
We will use again that the Polish topology in $\tilde{G}$ is the product topology
and that the uniform convergence topology is also a product topology (see Theorem \ref{unifmetric}). 
Let $g\in G$ be such that 
$\overline{g^G}^{d_u}$ is comeager, so $g$ is a metric generic. Let $T\in \Aut$ be aperiodic, so it is a metric generic in the space of invertible measure preserving transformation. We will prove that $(C_g,T)$ is a metric generic in $\tilde{G}$.

Notice that $(C_g,T)^{\tilde{G}}$ is Borel as well as $O=\overline{(C_g,T)^{\tilde{G}}}^{L_u}$. By Proposition
\ref{orbitaperiodic}, for $S$ in the comeager subset $\{S\in \Aut:S$ is aperiodic$\}$ of $\Aut$, $(C_g,S)\in O$. 
Thus by Kuratowski-Ulam and Observation \ref{doubleconj}, it suffices to check that the fibers 
$\pi_1\{\overline{(C_g,S)^{(f,Id)}}^{Lu}:f\in L^0([0,1],G)\}$ are comeager for all such $S$.

First notice that by Theorem \ref{unifmetric} and Section \ref{KlM3}, the set
$F=\overline{C_g^{L^0([0,1],G)}}^{\hat d_u}$ is comeager in $L^0([0,1],G)$. Now consider $\{(C_g,S)^{(f,Id)}:f\in L^0([0,1],G)\}$. Then by Theorem \ref{unifmetric} we get that $$\overline{\{(C_g,S)^{(f,Id)}:f\in L^0([0,1],G)\}}^{L_u}=(\overline{\{C_g^{f}:f\in L^0([0,1],G)\}}^{\hat d_u},S)=(F,S)$$ 
The result follows from the fact that $\{S\in \Aut:S$ is aperiodic$\}$ and $F$ are comeager.
\end{proof}

\begin{thm} Assume $(G,\tau,d_u)$ has metric ample generics, then so does $\tilde{G}$.
\end{thm}

\begin{proof}
The proof is very similar to the one of the previous theorem. 
For each $n\geq 1$ let $\vec g=(g_1,\dots,g_n)\in G^n$ be such that $\overline{{\vec g}^G}^{d_u}$ is comeager, where the action by $G$ is given by diagonal conjugation. 
Let $\vec T=(T_1,\dots,T_n)\in \Aut^n$ be a tuple such that $\overline{{\vec T}^{Aut([0,1])}}^{\Delta_u}$ is comeager, where again the action is given by diagonal conjugation. Then $(C_{g_1},\dots,C_{g_n},T_1,\dots,T_n)$ 
is a metric generic under the action by diagonal conjugation by $\tilde{G}$.
\end{proof}

\section{Extreme amenability}\label{EAm}

A topological group $G$ is \textbf{extremely amenable} if every action of $G$ on a compact space has a fixed point. Notice however that discrete groups $G$ have a free action on $\beta G$. Likewise, it can be shown that locally compact groups have a free action on some compact space. So this notion, although inspired in amenability on locally compact groups is only interesting on non-locally compact groups. 

\begin{thm}\label{Pestov} \begin{enumerate}
	\item (Pestov, Schneider, \cite{PS}) If $G$ is an amenable group, then $\Lcero$ is extremely amenable.
	\item (Giordano, Pestov, \cite{GP}) $\Aut$ is extremely amenable.
	\item (Pestov, see \cite{Pe}) Let $H'$ be a closed subgroup of a topological group $H$. If the topological groups $H'$ and $H/H'$ are extremely amenable, then so is $G$.
\end{enumerate}
\end{thm}

As a corollary of this, we obtain the following.

\begin{thm} If $G$ is an amenable group, then $\tilde{G}$ is extremely amenable. 
\end{thm}

\begin{proof} If $e$ is the identity on $G$, let $C_e$ denote the constant function with 
value $e$. Take $H=\{(C_e,S)| S\in \Aut\}$. This is a closed subgroup of $\tilde{G}$ which is extremely amenable as a topological group. Note that $G/H$ isomorphic as a topological group to $\Lcero$. In fact, take the function $F:\Lcero \to G/H$ defined by $F(f)=(f,\mbox{id})H$. Since $G$ is amenable by Theorem \ref{Pestov} we get that 
$L^0([0,1],G)$ is extremely amenable and again by Theorem \ref{Pestov} $\tilde{G}$ is extremely amenable. 
\end{proof}

\section{Questions}\label{Questions}

Here are a few open questions that follow from the results in these papers.

\begin{enumerate}
\item Does $\tilde S_\infty$ have automatic continuity?

\item If $G$ has Rohklin property does it have Rohklin property under powers?

\item If $H$ and $G$ have the strong Rohklin property, does $G\rtimes H$ have the strong Rohklin property?

\end{enumerate}

\bibliographystyle{alpha} %Style of Bibliography: plain / apalike / amsalpha / ...
\bibliography{Biblio} %You need a file 'literature.bib' for this.

\end{document}